\documentclass[11pt]{article}

\usepackage{amsmath, amssymb, amsthm, verbatim,enumerate,bbm,color}
\numberwithin{equation}{section}

\usepackage{mathrsfs}
\usepackage{enumitem}

\usepackage[backref,colorlinks,citecolor=blue,bookmarks=true]{hyperref}

\usepackage{tikz,mathtools}
\usepackage[capitalize]{cleveref}

\title{An efficient asymmetric removal lemma and its limitations}

\author{Lior Gishboliner\thanks{Department of Mathematics, University of Toronto, Toronto, Canada. Email: \texttt{lior.gishboliner@utoronto.ca}. Research was supported in part by SNSF grant 200021\_196965.} \and Asaf Shapira\thanks{School of Mathematics, Tel Aviv University, Tel Aviv, Israel. Email: \texttt{asafico@tau.ac.il}. Research was supported in part by ERC Consolidator Grant 863438 and NSF-BSF Grant 20196.} \and
Yuval Wigderson\thanks{Institute for Theoretical Studies, ETH Z\"urich, Z\"urich Switzerland. Email: \texttt{yuval.wigderson@eth-its.ethz.ch}. Research was supported in part by ERC Grants 863438 and 101044123, NSF-BSF Grant 20196, Dr.\ Max R\"ossler, the Walter Haefner Foundation, and the ETH Z\"urich Foundation.}}

\date{}
\allowdisplaybreaks

\theoremstyle{plain}
\newtheorem{theorem}{Theorem}[section]
\newtheorem{lemma}[theorem]{Lemma}

\newtheorem{proposition}[theorem]{Proposition}

\newtheorem{corollary}[theorem]{Corollary}

\newtheorem{conjecture}[theorem]{Conjecture}
\newtheorem{problem}[theorem]{Problem}
\theoremstyle{definition}
\newtheorem{remark}[theorem]{Remark}
\newtheorem{definition}[theorem]{Definition}

\newcommand{\Bin}{\ensuremath{\textrm{Bin}}}

\makeatletter
\def\moverlay{\mathpalette\mov@rlay}
\def\mov@rlay#1#2{\leavevmode\vtop{%
   \baselineskip\z@skip \lineskiplimit-\maxdimen
   \ialign{\hfil$\m@th#1##$\hfil\cr#2\crcr}}}
\newcommand{\charfusion}[3][\mathord]{
    #1{\ifx#1\mathop\vphantom{#2}\fi
        \mathpalette\mov@rlay{#2\cr#3}
      }
    \ifx#1\mathop\expandafter\displaylimits\fi}
\makeatother

\makeatletter
\renewenvironment{proof}[1][\proofname]
{\par\pushQED{\qed}
	\normalfont\topsep6\p@\@plus6\p@\relax\trivlist
	\item[\hskip\labelsep\bfseries#1\@addpunct{.}]
	\ignorespaces}
{\popQED\endtrivlist\@endpefalse}
\makeatother

\newcommand{\A}{\mathcal A}
\newcommand{\B}{\mathcal B}

\newcommand{\C}{\mathcal C}
\newcommand{\E}{\mathbb E}

\DeclareMathOperator\RS{RS}
\DeclareMathOperator\pr{Pr}
\newcommand\ab[1]{\lvert#1\rvert}

\newlist{lemenum}{enumerate}{1}
\setlist[lemenum]{label=(\alph*), ref=\thelemma(\alph*)}
\crefalias{lemenumi}{lemma}

\newlist{thmenum}{enumerate}{1}
\setlist[thmenum]{label=(\roman*), ref=\thetheorem(\roman*)}
\crefalias{thmenumi}{theorem}

\newlist{propenum}{enumerate}{1}
\setlist[propenum]{label=(\roman*), ref=\theproposition(\roman*)}
\crefalias{propenumi}{proposition}

\DeclareMathOperator{\poly}{poly}

\begin{document}
\maketitle

\begin{abstract}
The triangle removal states that if $G$ contains $\varepsilon n^2$ edge-disjoint triangles, then $G$
contains $\delta(\varepsilon)n^3$ triangles. Unfortunately, there are no sensible bounds on the order of growth
of $\delta(\varepsilon)$, and at any rate, it is known that $\delta(\varepsilon)$ is not polynomial in $\varepsilon$.
Csaba recently obtained an \emph{asymmetric} variant of the triangle removal, stating
that if $G$ contains $\varepsilon n^2$ edge-disjoint triangles, then $G$
contains $2^{-\poly(1/\varepsilon)}\cdot n^5$ copies of $C_5$. To this end, he devised a new variant of Szemer\'edi's regularity lemma.
We obtain the following results:
\begin{itemize}
\item We first give a regularity-free proof of Csaba's theorem, which improves the number of copies of $C_5$ to the optimal number $\poly(\varepsilon)\cdot n^5$.
\item We say that $H$ is {\em $K_3$-abundant} if every graph containing $\varepsilon n^2$ edge-disjoint triangles has
$\poly(\varepsilon)\cdot n^{\ab{V(H)}}$ copies of $H$. It is easy to see that a $K_3$-abundant graph must be triangle-free and
tripartite. Given our first result, it is natural to ask if all triangle-free tripartite
graphs are $K_3$-abundant.
Our second result is that assuming a well-known conjecture of Ruzsa in additive number theory,
the answer to this question is negative.
\end{itemize}
Our proofs use a mix of combinatorial, number-theoretic, probabilistic, and Ramsey-type arguments.

\end{abstract}
\newpage
\section{Introduction}\label{sec:intro}
\subsection{Background and previous results}
An $n$-vertex graph is said to be \emph{$\varepsilon$-far from triangle-free} if it cannot be made triangle-free by deleting fewer than $\varepsilon n^2$ edges\footnote{Note that if $G$ contains $\varepsilon n^2$ edge-disjoint triangles, then it is $\varepsilon$-far from triangle-free, as one must delete at least one edge from each triangle in this collection. Conversely, if $G$ is \emph{not} $(3\varepsilon)$-far from triangle-free, then any maximal collection of edge-disjoint triangles has less than $3 \varepsilon n^2$ edges, and thus $G$ does not contain $\varepsilon n^2$ edge-disjoint triangles. Therefore, the two notions are equivalent up to a constant factor, and we freely move between them throughout this paper.}.
A fundamental result in extremal graph theory is the triangle removal lemma of Ruzsa and Szemer\'edi \cite{RS}, which states that if an $n$-vertex graph is $\varepsilon$-far from triangle free, then it contains at least $\delta n^3$ triangles, where $\delta=\delta(\varepsilon)>0$ depends only on $\varepsilon$. Despite its simple statement, this is a deep result, and it has applications in graph theory, number theory, and theoretical computer science. For more on the triangle removal lemma, we refer to the survey \cite{CF}.

Despite decades of study, very little is known about the quantitative behavior of $\delta(\varepsilon)$ in the triangle removal lemma. The original proof of Ruzsa and Szemer\'edi used Szemer\'edi's regularity lemma, and consequently obtained a bound on $1/\delta$ that was of tower type in $\poly(1/\varepsilon)$. The best known bound is due to Fox \cite{Fox}, who improved the bound on $1/\delta$ to a tower of height $O(\log(1/\varepsilon))$ by finding a new proof which avoids the use of the regularity lemma. While a major improvement, this bound is still enormous, and it is a major open problem to improve it further.

In the other direction, the best known lower bound on $1/\delta$ is due to Ruzsa and Szemer\'edi \cite{RS}, who proved that $1/\delta \geq (1/\varepsilon)^{\Omega(\log(1/\varepsilon))}$. In particular, this implies that $\delta(\varepsilon)$ can not be taken to be a polynomial function of $\varepsilon$. Ruzsa and Szemer\'edi proved this by relating the triangle removal lemma to a problem in additive combinatorics, namely the problem of finding subsets of $[n]$ containing no three-term arithmetic progression, and using a well-known construction of Behrend \cite{Behrend} of such a subset.

If one examines the usual proof of the triangle removal lemma, one sees that it immediately gives the following stronger ``asymmetric'' statement (see e.g. \cite[Theorem 1.13]{FZ}).
\begin{theorem}\label{thm:homomorphism-removal}
	For every graph $H$ with $\chi(H) \leq 3$ and every $\varepsilon>0$, there exists $\delta = \delta(\varepsilon,H)>0$ such that the following holds. If an $n$-vertex graph $G$ is $\varepsilon$-far from triangle-free, then it contains at least $\delta n^{\ab{V(H)}}$ copies of $H$.
\end{theorem}
It is easy to see that the assumption $\chi(H)\leq 3$ is necessary in \cref{thm:homomorphism-removal}, as a balanced complete tripartite graph on $n$ vertices has no copies of any $H$ with $\chi(H)>3$, but is $\frac 19$-far from triangle-free.

The standard proof of \cref{thm:homomorphism-removal} gives tower-type bounds on $\delta$, as it uses Szemer\'edi's regularity lemma. But it is natural to ask whether, for certain graphs $H$, one can get a stronger bound in \cref{thm:homomorphism-removal}. It turns out that in certain instances, one can.
\begin{theorem}[Csaba {\cite[Theorem 5.2]{Csaba}}]\label{thm:csaba}
	If an $n$-vertex graph $G$ is $\varepsilon$-far from triangle-free, then it contains at least $2^{-\poly(1/\varepsilon)}\cdot n^5$ copies of $C_5$.
\end{theorem}
In other words, Csaba showed that in the special case $H=C_5$, one can replace the tower-type bound in \cref{thm:homomorphism-removal} by a single-exponential bound. To prove this, Csaba developed a new variant of Szemer\'edi's regularity lemma, somewhat akin to the weak regularity lemmas of Frieze--Kannan \cite{FK} and Duke--Lefmann--R\"odl \cite{DLR}, which does not yield tower-type dependencies between the parameters and is nonetheless strong enough to prove results like \cref{thm:csaba}. We also mention a recent result of Conlon, Fox, Sudakov, and Zhao \cite[Corollary 1.4]{CFSZ}, who proved that if $G$ has \emph{zero} copies of $C_5$, then it can be made triangle-free by deleting $o(n^{3/2})$ edges; this result was proved via new techniques in the regularity method for sparse graphs.

\subsection{An optimal result for odd cycles}
Our first main result is an improvement of \cref{thm:csaba}, which reduces the single-exponential bound to a polynomial bound. In fact, we prove a more general result, which holds for all pairs of odd cycles; our improvement to \cref{thm:csaba} corresponds to the case $k=1,\ell=2$ of the following theorem.
Extending the terminology above, one says that an $n$-vertex graph $G$ is \emph{$\varepsilon$-far} from satisfying a graph property $\mathcal P$ if one must add or delete at least $\varepsilon n^2$ edges to $G$ in order to create a graph satisfying property $\mathcal P$.
\begin{theorem}\label{thm:odd-cycle-abundant}
	Let $1 \leq k <\ell$ be integers and let $\varepsilon>0$. If $G$ is $\varepsilon$-far from $C_{2k+1}$-free, then it contains at least $(c_\ell \varepsilon^{4\ell+2})n^{2\ell+1}$ copies of $C_{2\ell+1}$, where $c_\ell>0$ is a constant depending only on $\ell$.
\end{theorem}
In contrast to the earlier proofs of \cref{thm:homomorphism-removal,thm:csaba}, which used complicated regularity-type arguments, our proof of \cref{thm:odd-cycle-abundant} uses elementary but subtle counting arguments.
Note that the dependence on $\varepsilon$ in \cref{thm:odd-cycle-abundant} is best possible up to a factor of $2$ in the exponent. Indeed, a random $n$-vertex graph with edge density $\varepsilon$ has $\Theta_\ell(\varepsilon^{2\ell+1})\cdot n^{2\ell+1}$ copies of $C_{2\ell+1}$, and is $\Theta_k(\varepsilon)$-far from $C_{2k+1}$-free. This observation, as well as \cref{thm:odd-cycle-abundant}, motivates the following definition.

\begin{definition}[$K_3$-abundant]\label{def:k3-abundant}
	Let $H$ be a graph. We say that $H$ is \emph{$K_3$-abundant} if there exists some $C_H>0$ such that for all $0<\varepsilon \leq \frac 12$, all integers $n$, and any $n$-vertex graph $G$ which is $\varepsilon$-far from triangle-free, the number of copies of $H$ in $G$ is at least $\varepsilon^{C_H} \cdot n^{\ab{V(H)}}$.
\end{definition}
Informally, this definition says that $H$ is {$K_3$-abundant} if we may take $\delta = \poly_H(\varepsilon)$ in \cref{thm:homomorphism-removal}. In this language, \cref{thm:odd-cycle-abundant} implies that $C_{2\ell+1}$ is $K_3$-abundant for all $\ell \geq 2$.

By the discussion above, any graph $H$ with $\chi(H)>3$ cannot be $K_3$-abundant, as \cref{thm:homomorphism-removal} is simply false for such graphs. At the other extreme,
it is easy to see that every bipartite graph $H$ is $K_3$-abundant, but for a fairly uninteresting reason: if an $n$-vertex graph $G$ contains fewer than $\varepsilon^{C_H} \cdot n^{\ab{V(H)}}$ copies of some bipartite $H$, then $G$ has fewer than $\varepsilon n^2$ edges, and thus is certainly not $\varepsilon$-far from triangle-free. This follows from a convexity argument originally due to K\H ov\'ari, S\'os, and Tur\'an \cite{KST} (see also \cite{Alon}).

On the other hand, $K_3$ itself is not $K_3$-abundant, thanks to the Ruzsa--Szemer\'edi result that one does not have polynomial bounds in the triangle removal lemma. Moreover, if $H$ is $K_3$-abundant, then so is any subgraph of it. This immediately implies that if $H$ contains a triangle, then $H$ cannot be $K_3$-abundant. So any non-bipartite $K_3$-abundant graph must be tripartite and triangle-free. The simplest examples of such graphs are odd cycles (of length at least $5$), which are indeed $K_3$-abundant by \cref{thm:odd-cycle-abundant}. Moreover, it is easy to see\footnote{Indeed, if $H_1$ is homomorphic to $H_2$ and an $n$-vertex graph $G$ contains $\delta n^{\ab{V(H_2)}}$ copies of $H_2$, then it contains $\poly(\delta)n^{\ab{V(H_1)}}$ copies of $H_1$, which implies that $H_1$ is $K_3$-abundant if $H_2$ is. This follows from the hypergraph analogue of the K\H ov\'ari--S\'os--Tur\'an theorem due to Erd\H os \cite{Erdos}.} that any graph which is homomorphic to a $K_3$-abundant graph is also $K_3$-abundant, hence any graph homomorphic to an odd cycle of length at least $5$ is $K_3$-abundant.

Based on all of this, it is natural to ask whether \emph{all} triangle-free tripartite graphs are $K_3$-abundant.
\subsection{Not all graphs are abundant}
Our second main result is that not all triangle-free tripartite graphs are $K_3$-abundant, assuming a well-known conjecture of Ruzsa \cite{Ruzsa} in additive combinatorics. To state this conjecture, we need to set up some notation. Consider a linear equation $\sum_{i=1}^k a_i x_i = 0$, where the coefficients $a_1,\dots,a_k$ are integers. Following Ruzsa \cite{Ruzsa}, one says that this equation has \emph{genus one} if $\sum_{i=1}^k a_i=0$, but $\sum_{i \in T} a_i \neq 0$ for all $\varnothing \subsetneq T \subsetneq [k]$.
We say that integers $y_1,\dots,y_k$ form a \emph{non-trivial solution\footnote{For general equations, there are other types of trivial solutions, but these are the only ones one needs to consider if an equation has genus one.}} to this equation if $\sum_{i=1}^k a_{i} y_i=0$ and the $y_i$ are not all equal. We denote by $r_E(m)$ the maximum size of a subset of $[m]$ containing no non-trivial solution to an equation $E$.

With this notation in place, Ruzsa's genus conjecture\footnote{In fact, Ruzsa formulated a more general conjecture, predicting the rough behavior of $r_E(m)$ for any equation $E$, as a function of its genus. But we do not need this more general statement.} says the following.
\begin{conjecture}[Ruzsa \cite{Ruzsa}]\label{conj:ruzsa}
  If $E$ is a linear equation of genus one, then $r_E(m) \geq m^{1-o(1)}$.
\end{conjecture}

We can now state our second main theorem.

\begin{theorem}\label{thm:main-non-abundant}
	Assuming that \cref{conj:ruzsa} holds, there exists a triangle-free tripartite graph $H$ which is not $K_3$-abundant.
\end{theorem}
The proof of \cref{thm:main-non-abundant} is based on Ruzsa and Szemer\'edi's proof \cite{RS} that super-polynomial bounds are necessary in the triangle removal lemma (i.e.\ that $K_3$ itself is not $K_3$-abundant), but with several important differences. To explain the idea of the proof, let us briefly recall Ruzsa and Szemer\'edi's construction. Given an integer $m$ and a set $R \subseteq [m]$, they define a tripartite graph $\Gamma$ such that triangles in $\Gamma$ correspond to solutions in $R$ to the equation $x+z=2y$; note that solutions to this equation are three-term arithmetic progressions in $R$, and that trivial solutions are trivial arithmetic progressions. The upshot of this construction is twofold. First, if $R$ contains no non-trivial arithmetic progressions, then $\Gamma$ has few triangles, and second, if $R$ is large, then $\Gamma$ is far from triangle-free (as the many trivial arithmetic progressions in $R$ yield many edge-disjoint triangles in $\Gamma$). To finish the proof, Ruzsa and Szemer\'edi use Behrend's \cite{Behrend} construction of a set $R\subseteq [m]$ with no non-trivial arithmetic progressions and of size $\ab R \geq m^{1-o(1)}$. Using this set $R$ in the construction described above, one obtains a graph $\Gamma$ with at most $\delta n^3$ triangles that is $\varepsilon$-far from triangle-free, where $\delta$ is smaller than any fixed power of $\varepsilon$, because $\ab R$ is larger than any fixed power of $m$ less than $1$.

For our proof of \cref{thm:main-non-abundant}, we use the same construction to build a graph $\Gamma$ out of any $R \subseteq [m]$. For the same reason as above, if $R$ is large, then $\Gamma$ is far from triangle-free. We would now like to show that $\Gamma$ has few copies of $H$, for some triangle-free tripartite graph $H$. As before, we can use the structure of $\Gamma$ to parameterize copies of $H$ in $\Gamma$: there is a correspondence between copies of $H$ in $\Gamma$ and solutions in $R$ to a certain system of equations $S$ arising from the cycles in $H$, whose variables are indexed by the edges of $H$. To prove \cref{thm:main-non-abundant}, it suffices to find a set $R \subseteq [m]$ with $\ab R \geq m^{1-o(1)}$ containing no non-trivial solutions to $S$. Unfortunately, as $H$ is triangle-free, this is impossible if we only use a \emph{single} equation from $S$: no single equation in this family has genus one, and therefore the largest set $R \subseteq [m]$ avoiding non-trivial solutions to any single equation from $S$ has size $\ab R = O(\sqrt m)$ \cite[Theorem 3.6]{Ruzsa}. The key idea in the proof is that if $H$ is sufficiently ``complicated'', then the system of equations $S$, generated from the cycle space of $H$, does have genus one\footnote{We formally define what it means for a system of equations (rather than one equation) to have genus one in \cref{sec:non-abundant-proof}.}, and thus \cref{conj:ruzsa} implies that there exists a set $R$ satisfying our desired properties.

The heart of the proof, then, is showing that if $H$ is an appropriately chosen tripartite triangle-free graph, the set of equations arising from its cycles has genus one. Since the variables in the equations in $S$ correspond to the edges of $H$, a set which witnesses that $S$ does not have genus one can be viewed as a two-coloring of the edges of $H$ with a certain properties. We now use a Ramsey-theoretic argument to show that if $H$ satisfies a number of carefully chosen pseudorandomness conditions, such a coloring cannot exist. To conclude the proof, it suffices to find a tripartite triangle-free graph that is pseudorandom in this sense; we do this by picking a random tripartite graph of an appropriate density and deleting one edge from each triangle.

\subsection{An application of Theorem \ref{thm:odd-cycle-abundant} to property testing}

Recall that an $n$-vertex graph $G$ is \emph{$\varepsilon$-far} from satisfying a graph property $\mathcal P$ if one must add or delete at least $\varepsilon n^2$ edges to $G$ in order to create a graph satisfying property $\mathcal P$.

Given a monotone\footnote{A graph property is called \emph{monotone} if it is closed under removal
of vertices and edges.} graph property ${\cal P}$, let $w_{\cal P}(\varepsilon)$ denote the smallest integer so that if $G$ is $\varepsilon$-far
from satisfying ${\cal P}$, then a randomly selected set $X$ of $w_{\cal P}(\varepsilon)$ vertices spans a subgraph not satisfying ${\cal P}$
with probability at least $1/2$. Note that a priori it is not clear that the function $w_{\cal P}(\varepsilon)$ is well-defined for all (or indeed any) ${\cal P}$. In this terminology, if ${\cal P}$ is the property of being triangle-free, then the triangle removal lemma 
implies that $w_{\cal P}(\varepsilon)$ is well-defined, and that in fact $w_{\cal P}(\varepsilon) \leq 1/ \delta(\varepsilon)$.

As discussed above, Ruzsa and Szemer\'edi \cite{RS} proved that the bounds for the triangle removal lemma are not
polynomial in $\varepsilon$. In the notation of the previous paragraph, this means that $w_{\cal P}(\varepsilon)$ is super-polynomial
in $1/\varepsilon$ when $\mathcal P$ is the property of being triangle-free. Alon \cite{Alon} later extended this result to all non-bipartite graphs $H$, and thus in particular
to all odd cycles: if $\mathcal P$ is the property of being $C_{2\ell+1}$-free, then $w_{\mathcal P}(\varepsilon)$ is super-polynomial in $1/\varepsilon$. It is natural to ask at this point what happens if instead of forbidding a single odd cycle, we take
a family of odd cycles $L$ and consider the property of not containing any cycle from $L$. It is not hard to show that Alon's
method can be used to show that $w_{\cal P}(\varepsilon)$ is super-polynomial
in $1/\varepsilon$ for every {\em finite} family of odd cycles. At the other extreme, if we take $L$ to be the family
of all odd cycles, then we have the following influential result of Goldreich, Goldwasser, and Ron \cite{GGR}.

\begin{theorem}[\cite{GGR}]\label{thmggg} 
If $\mathcal P$ is the property of being $2$-colorable, then
$w_{\mathcal P}(\varepsilon) \leq \poly(1/ \varepsilon)$.
\end{theorem}
The $\poly(1/\varepsilon)$ bound obtained in \cite{GGR} improved a tower-type bound obtained by
Bollob\'{a}s, Erd\H{o}s, Simonovits and Szemer\'{e}di \cite{BESS}. It is interesting to
compare the proofs in \cite{BESS} and \cite{GGR}. The former proof relied on the fact
that if $G$ is far from being bipartite then $G$ is also far from being $C_{2\ell+1}$-free
where $\ell$ is an integer of order $1/\varepsilon$. They then used the graph removal lemma to show that $G$ must contain many copies of $C_{2\ell+1}$.
The proof in \cite{GGR} used a completely different approach, which crucially relied on interpreting this property as bipartiteness, rather than $\{C_3,C_5,\dots\}$-freeness. In particular, the proof of \cite{GGR} works not only for bipartiteness but for $k$-colorability for any fixed $k$.

It is natural at this point to ask what happens for general infinite families of odd cycles.
We answer this question in Section \ref{sec:odd-cycles}. In particular, our proof shows
that one can prove Theorem \ref{thmggg} using a removal-type argument similar to that of
\cite{BESS}, but with an important twist. After concluding that $G$ is far from being $C_{2\ell+1}$-free,
we do not prove that it has many copies of $C_{2\ell+1}$, but rather that it has many copies of $C_{2\ell+3}$. Thanks to \cref{thm:odd-cycle-abundant}, we may take this ``many'' to only be polynomial in $1/\varepsilon$, rather than super-polynomial as in \cite{BESS}.

It would be very interesting to prove efficient asymmetric removal lemmas like \cref{thm:odd-cycle-abundant} for pairs of graphs of chromatic number larger than $3$, and then use them to prove a version of \cref{thmggg} which holds for $k$-colorability for any fixed $k$. We discuss a version of this problem in the next subsection.

\subsection{Larger cliques and higher chromatic numbers}
There is a natural generalization of the triangle removal lemma, known as the \emph{clique removal lemma}, which states that for any $t \geq 3$, if an $n$-vertex graph is $\varepsilon$-far from $K_t$-free, then it contains at least $\delta n^t$ copies of $K_t$, for some $\delta = \delta(\varepsilon,t)>0$. Similarly to \cref{thm:homomorphism-removal}, the usual proof of the clique removal lemma immediately gives the following asymmetric version. 
\begin{theorem}\label{thm:asymmetric-clique-removal}
	Let $t \geq 3$ and let $H$ be a graph with $\chi(H) \leq t$. If an $n$-vertex graph $G$ is $\varepsilon$-far from $K_t$-free, then $G$ contains at least $\delta n^{\ab{V(H)}}$ copies of $H$, for some $\delta = \delta(\varepsilon,H,t)>0$.
\end{theorem}

Because of this, one can naturally extend the definition of $K_3$-abundance to $K_t$-abundance for any $t \geq 3$. Namely, we say that $H$ is $K_t$-abundant if we may take $\delta = \poly_{H,t}(\varepsilon)$ in \cref{thm:asymmetric-clique-removal}. Note that if $G$ is $\varepsilon$-far from $K_t$-free for some $t \geq 4$, then $G$ is also $\varepsilon$-far from triangle-free. Therefore, if $H$ is $K_3$-abundant, then it is automatically $K_t$-abundant for all $t \geq 4$. This shows that all bipartite graphs, as well as all odd cycles of length at least $5$, are $K_t$-abundant for all $t \geq 4$. 

In order to rule out such examples, it is natural to ask about $K_t$-abundant graphs with chromatic number equal to $t$. As in the case of $K_3$-abundance, one can use the Ruzsa--Szemer\'edi argument to show that if $H$ contains a triangle, then $H$ is not $K_t$-abundant for any $t \geq 3$ (see e.g.\ \cite[Lemma 4.2]{AS} for a proof of a very similar result).

However, once $t \geq 4$, there are other ``simple'' reasons why a graph may be non-$K_t$-abundant. To define these, let $H$ be a graph with $\chi(H)=t$. Given a proper coloring $c:V(H) \to [t]$, let us define a \emph{$c$-increasing cycle} to be a cycle $v_1,\dots,v_s$ in $H$ with $c(v_1)< c(v_2) < \dotsb < c(v_s)$. Note that a $c$-increasing cycle necessarily has length at most $t$. Let us say that $H$ is \emph{increasing-cycle-unavoidable} if, for every proper coloring $c:V(H) \to [t]$, there is a $c$-increasing cycle in $H$. For example, if $H$ contains a triangle, then it is necessarily increasing-cycle-unavoidable.

It again follows from standard techniques (e.g.\ by combining \cite[Theorem 2.3]{Ruzsa} and \cite[Lemma 3.4]{Alon}) that if $H$ has chromatic number $t$ and is increasing-cycle-unavoidable, then $H$ is not $K_t$-abundant. In particular, we again find that any graph containing a triangle is not $K_t$-abundant. However, this also includes other graphs. For example, the Gr\"otzsch graph is the smallest triangle-free $4$-chromatic graph, and one can check (by tedious casework, or with a computer), that the Gr\"otzsch graph is increasing-cycle-unavoidable. Therefore, the Gr\"otzsch graph is an example of a triangle-free $4$-chromatic graph that is not $K_4$-abundant. 

Therefore, to get to a genuinely interesting question, it makes sense to restrict our attention to $t$-chromatic graphs with girth greater than $t$ (as any $c$-increasing cycle has length at most $t$, so there are no such cycles in a graph of girth larger than $t$). We believe that, with appropriate modifications, the technique used to prove \cref{thm:main-non-abundant} shows that for every $t\geq 3$, there exists a $t$-chromatic non-$K_t$-abundant graph with girth greater than $t$, assuming that \cref{conj:ruzsa} holds. Namely, one can obtain such a graph by sampling an appropriate random graph, then deleting all cycles of length at most $t$. 

Thus, we have many ways of finding non-$K_t$-abundant graphs, and no examples of $t$-chromatic graphs that \emph{are} $K_t$-abundant for $t \geq 4$. This motivates the following question.
\begin{problem}\label{prob:K_t-abundant}
	Let $t \geq 4$. Does there exist a $K_t$-abundant graph with chromatic number $t$?
\end{problem}
The simplest $4$-chromatic graph whose status is unknown is the Brinkmann graph; this is a $21$-vertex graph with chromatic number $4$ and girth $5$, so the techniques we have cannot be used to rule out its $K_4$-abundance.

\paragraph{Paper organization:} The rest of this paper is organized as follows. In \cref{sec:odd-cycles}, we prove \cref{thm:odd-cycle-abundant} and use it to deduce \cref{thmggg}. In \cref{sec:non-abundant-proof}, we prove \cref{thm:main-non-abundant}, apart from the proof that there exists a graph satisfying certain pseudorandomness assumptions.
In \cref{sec:random-graphs}, we prove that an appropriate random graph satisfies these assumptions. We end in \cref{sec:conclusion} with some concluding remarks.

\section{Many edge-disjoint \texorpdfstring{$(2k+1)$}{(2k+1)}-cycles imply many \texorpdfstring{$(2\ell+1)$}{(2l+1)}-cycles }\label{sec:odd-cycles}

In this section, we prove \cref{thm:odd-cycle-abundant}. We actually prove a stronger sampling version, saying that a sample of size $\poly(\ell/\varepsilon)$ contains a copy of $C_{2\ell+1}$ with high probability. 

\begin{lemma}\label{lem:odd cycles}
	Let $1 \leq k < \ell$, let $\varepsilon > 0$ and let $G$ be a graph on $n \geq 200\ell k^2/\varepsilon^2$ vertices containing a collection $\mathcal{C}$ of $\varepsilon n^2$ edge-disjoint copies of $C_{2k+1}$. Then a sample of $\frac{100k^2(2\ell+1)\log(10\ell)}{\varepsilon^2}$ vertices of $G$ (taken uniformly and independently) contains a copy of $C_{2\ell+1}$ with probability at least $\frac{2}{3}$.
\end{lemma}
Note that, as discussed above, the properties of being $\varepsilon$-far from $C_{2k+1}$-free and having $\varepsilon n^2$ edge-disjoint copies of $C_{2k+1}$ are equivalent up to a constant factor.
\begin{proof}[Proof of \cref{lem:odd cycles}]
	There exists a collection $\mathcal{C}_0 \subseteq \mathcal{C}$ such that $\ab{\mathcal{C}_0} \geq \varepsilon n^2/2$ and each vertex $v \in V(G)$ belongs to either $0$ or at least $\varepsilon n/2$ of the cycles in $\mathcal{C}_0$. Indeed, to obtain $\mathcal{C}_0$, we repeatedly delete from $\mathcal{C}$ all cycles containing a vertex $v$ which belongs to fewer than $\varepsilon n/2$ of the cycles in $\mathcal{C}$ (without changing the graph). The set of cycles left at the end is $\mathcal{C}_0$. In this process, we delete at most $\varepsilon n^2/2$ cycles altogether (because each vertex is contained in at most $n$ cycles from $\mathcal{C}$); hence $\ab{\mathcal{C}_0} \geq \varepsilon n^2/2$.
	Let $V_0$ be the set of vertices contained in at least $\varepsilon n/2$ cycles from $\mathcal{C}_0$. Then $|V_0| \geq \varepsilon n$ because the cycles containing a given $v \in V_0$ are edge-disjoint and contained in $V_0$.
	
	Set $q := \frac{100k^2\log(10\ell)}{\varepsilon^2}$. We sample $2\ell+1$ sets $S_0,\dots,S_{2\ell}$ of size $q$ each, where each $S_i$ contains $q$ vertices sampled uniformly at random and independently with repetition\footnote{i.e.\ it is possible that $\ab{S_i}<q$ for some $i$, as we may sample the same vertex multiple times.}. The probability that $S_0 \cap V_0 = \varnothing$ is at most $(1 - \varepsilon)^{q} \leq e^{-\varepsilon q} \leq \frac{1}{10}$. From now on, we assume that $S_0 \cap V_0 \neq \varnothing$ and fix $v_0 \in S_0 \cap V_0$.
	Let $N$ be the set of vertices $u$ such that $v_0u \in E(C)$ for some $C \in \mathcal{C}_0$. Then $|N| \geq \varepsilon n$.
	Let $\mathcal{C}(v_0)$ be the set of cycles $C \in \mathcal{C}_0$ such that
	$V(C) \cap N \neq \varnothing$ and $v_0 \notin V(C)$. The number of cycles $C \in \mathcal{C}_0$ intersecting $N$ is at least $\varepsilon n/2 \cdot |N| /(2k+1) \geq \varepsilon^2 n^2/(4k+2)$, and the number of cycles containing $v_0$ is at most $n$. Hence,
	$|\mathcal{C}(v_0)| \geq \varepsilon^2 n^2/(4k+2) - n \geq \varepsilon^2 n^2/(7k)$, using our assumption that $n \geq 200 \ell k^2/\varepsilon^2$.
	
	We denote by $C_{2k+1}$ the labeled cycle on vertices $1,\dots,2k+1$. For each $C \in \mathcal{C}(v_0)$, fix an isomorphism $f_C : C_{2k+1} \rightarrow C$ mapping $1$ to a vertex in $V(C) \cap N$.
	For a vertex $u$ and $1 \leq j \leq 2k+1$, denote by $d_j(u)$ the number of cycles $C \in \mathcal{C}(v_0)$ with $f_C(j) = u$.
	As long as there is $1 \leq j \leq 2k+1$ and a vertex $u$ with
	$0 < d_j(u) < \varepsilon^2 n/(50k^2)$, then
	delete from $\mathcal{C}(v_0)$ all cycles $C$ with $f_C(j) = u$ (without changing the graph). In this process, we delete at most
	$(2k+1) \cdot n \cdot \varepsilon^2 n/(50k^2) \leq \varepsilon^2n^2/(14k)$ cycles. Let $\mathcal{C}^*(v_0)$ be the set of cycles left at the end of the process. Then $|\mathcal{C}^*(v_0)| \geq |\mathcal{C}(v_0)| - \varepsilon^2n^2/(14k) \geq \varepsilon^2n^2/(14k)$. At the end, we have that $d_j(u) = 0$ or $d_j(u) \geq \varepsilon^2 n/(50k^2)$ for every vertex $u$ and every $1 \leq j \leq 2k+1$. Let $U_j$ be the set of vertices $u$ with $d_j(u) \geq \varepsilon^2 n/(50k^2)$. Observe that each $u \in U_j$ has at least $\varepsilon^2 n/(50k^2)$ neighbors in $U_{j-1}$ and $U_{j+1}$, with indices taken modulo $2k+1$.
	In particular, $|U_j| \geq \varepsilon^2 n/(50k^2)$ for every $1 \leq j \leq 2k+1$.
	Also, $U_1 \subseteq N \subseteq N(v_0)$ by definition, where $N(v_0)$ denotes the neighborhood of $v_0$.
	
	Let $P_{2\ell-1}$ be the path with vertices $1,\dots,2\ell$. Fix a homomorphism $\varphi : P_{2\ell-1} \rightarrow C_{2k+1}$ such that $\varphi(1) = \varphi(2\ell) = 1$. This is possible because $2k+1$ is odd and $k < \ell$.
	We claim that with probability at least $\frac{4}{5}$, there are vertices $u_i \in S_i$, $1 \leq i \leq 2\ell$, such that $u_1,\dots,u_{2\ell}$ is a path and $u_i \in U_{\varphi(i)}$. Observe that if this happens then $v_0,u_1,\dots,u_{2\ell}$ is a copy of $C_{2\ell+1}$ because $u_1,u_{2\ell} \in U_1 \subseteq N(v_0)$.
	First, since $|U_1| \geq \varepsilon^2 n/(50k^2)$, the probability that $S_1$ contains no vertex of $U_1$ is at most $(1 - \frac{\varepsilon^2}{50k^2})^q \leq \frac{1}{10\ell}$. So suppose that there is $u_1 \in S_1 \cap U_1$. For $2 \leq i \leq 2\ell$, suppose that we already found $u_1,\dots,u_{i-1}$; in particular, $u_{i-1} \in U_{\varphi(i-1)}$. We saw that $u_{i-1}$ has at least $\varepsilon^2 n/(50k^2)$ neighbors in $U_{\varphi(i)}$. Out of these neighbors, at most $2\ell-1$ are on the path $u_1,\dots,u_{i-1}$. Hence, at least $\varepsilon^2 n/(50k^2) - 2\ell \geq \varepsilon^2n/(100k^2)$ are not in $\{u_1,\dots,u_{i-1}\}$.
	The probability that $S_i$ contains none of these neighbors is at most $(1 - \frac{\varepsilon^2}{100k^2})^q \leq \frac{1}{10\ell}$. So suppose that there is $u_i \in S_i \cap U_{\varphi(i)}$ and continue the process. The probability that this process fails to produce $u_1,\dots,u_{2\ell}$ is at most $2\ell \cdot \frac{1}{10\ell} = \frac{1}{5}$, as required. In total, the probability that $S_0 \cup \dots \cup S_{2\ell+1}$ contains no copy of $C_{2\ell+1}$ is at most $\frac{1}{10} + \frac{1}{5} < \frac{1}{3}$. This completes the proof.
\end{proof}
This immediately implies \cref{thm:odd-cycle-abundant}, which we restate as the following corollary.
\begin{corollary}
	Let $1 \leq k < \ell$, let $\varepsilon > 0$ and let $G$ be a graph on $n \geq 200\ell k^2/\varepsilon^2$ vertices with $\varepsilon n^2$ edge-disjoint copies of $C_{2k+1}$. Then $G$ has $\Omega_\ell(\varepsilon^{4\ell+2})n^{2\ell+1}$ copies of $C_{2\ell+1}$.
\end{corollary}
\begin{proof}
	Let $M$ denote the number of copies of $C_{2\ell+1}$ in $G$. The probability that a sample of size $q$ contains a  copy of $C_{2\ell+1}$ is at most $M \cdot (q/n)^{2\ell+1}$. By  \cref{lem:odd cycles}, this probability is at least $2/3$ for $q = O_\ell(\varepsilon^{-2})$. The result follows.
\end{proof}

We now turn to removal bounds for infinite families of odd cycles.
For a sequence of positive integers $L$, let $w_L(\varepsilon)$ be the smallest integer $q$ such that if a graph $G$ on $n \geq n_0(\varepsilon)$ vertices is $\varepsilon$-far from being $\{C_{\ell} : \ell \in L\}$-free, then with probability at least $2/3$, a sample of $q$ vertices from $G$ contains a copy of $C_{\ell}$ for some $\ell \in L$. A result of \cite{GS} gave tight bounds on $w_L(\varepsilon)$ as a function of the growth-rate of the sequence $L$. We state a variant of this result as follows.

\begin{theorem}\label{thm:hierarchy}
	Let $g : \mathbb{N}_{odd} \rightarrow \mathbb{N}_{odd}$ be an increasing function satisfying $g(x) > x$. Let $\ell_1 \geq 3$ be an odd number and define a sequence $(\ell_i)_{i \geq 1}$ inductively by setting $\ell_{i+1} = g(\ell_i)$. For $L = \{\ell_i : i \geq 1\}$, we have
	$w_{L}(\varepsilon) \leq O(\varepsilon^{-8}) \cdot g(4/\varepsilon) \cdot \log(g(4/\varepsilon))$ for every $0<\varepsilon \leq \frac{1}{4\ell_1}$. 
\end{theorem}

Note that \cref{thmggg} follows immediately from \cref{thm:hierarchy}, by setting $g(x)=x+2$ and $\ell_1=3$, so that $L$ consists of all odd integers greater than $1$. This implies that if $G$ is $\varepsilon$-far from being bipartite then a sample of size $\poly(1/\varepsilon)$ contains an odd cycle with probability at least $2/3$. We note that the dependence on $\varepsilon$ that we get is not tight; indeed, it is known \cite{AK} that a sample of $\tilde{O}(1/\varepsilon)$ suffices.

Here we give a new simpler proof of \cref{thm:hierarchy}. We need the following
easy fact, observed already in \cite{BESS}. For completeness, we give a proof.
\begin{lemma}\label{lem:shortest odd cycle}
	For every $0 < \varepsilon \leq 1/2$, every graph $G$ which is $\varepsilon$-far from being bipartite contains an odd cycle of length at most
	$2/\varepsilon$.
\end{lemma}
\begin{proof}
	Repeatedly delete from $G$ vertices of degree less than $\varepsilon n$, until no such vertices are left. Let $W$ be the set of remaining vertices. The total number of edges deleted in this process is less than $\varepsilon n^2$, hence $G[W]$ is not bipartite. Also, each vertex in $G[W]$ has degree at least $\varepsilon n$. Let $C$ be a shortest odd cycle in $G[W]$.
	If $C$ is a triangle then we are done, so suppose $|C| \geq 5$.
	By minimality, $C$ is an induced cycle.
	For $x,y \in V(C)$, denote by $\text{dist}_C(x,y)$ the distance between $x$ and $y$ along $C$.
	Observe that if $x,y \in V(C)$ have a common neighbor $v \in W \setminus V(C)$, then $\text{dist}_C(x,y) = 2$. Indeed, suppose otherwise, and let $P_1,P_2$ be the two paths between $x,y$ along $C$. Without loss of generality, suppose that $P_1$ is odd and $P_2$ is even. We have $|P_2| > 2$ because $\text{dist}_C(x,y) \neq 2$. Now, replacing $P_2$ with the path $xvy$, we obtain a shorter odd cycle, a contradiction. It follows that every $v \in W$ has at most $2$ neighbors on $C$. Indeed, if $x,y,z \in V(C)$ are neighbors of $v$, then any two among $x,y,z$ are at distance $2$, implying that $C$ has length $6$, a contradiction. Now, summing over the degrees of vertices in $C$, we get
	$
	|C| \cdot \varepsilon n \leq \sum_{x \in V(C)}d_W(x) \leq 2|W| \leq 2n.
	$
	It follows that $|C| \leq 2/\varepsilon$.
\end{proof}

\begin{proof}[Proof of \cref{thm:hierarchy}]
	Let $G$ be a graph which is $\varepsilon$-far from being $\{C_{\ell_i} : i \geq 1\}$-free.
	We claim that there is an odd integer $1 \leq 2k+1 \leq \frac{4}{\varepsilon}$ such that $G$ is $\frac{\varepsilon^2}{4}$-far from being $C_{2k+1}$-free. Indeed, otherwise we could destroy all odd cycles in $G$ of length at most $\frac{4}{\varepsilon}$ by deleting at most $\frac{2}{\varepsilon} \cdot \frac{\varepsilon^2}{4}n^2 = \frac{\varepsilon}{2}n^2$ edges. Let $G'$ be the resulting graph. By \cref{lem:shortest odd cycle}, $G'$ is not $\frac{\varepsilon}{2}$-far from being bipartite. This implies that $G$ is not $\varepsilon$-far from being $\{C_{\ell_i} : i \geq 1\}$-free, a contradiction.
	
	So fix $1 \leq 2k+1 \leq \frac{4}{\varepsilon}$ such that $G$ is $\frac{\varepsilon^2}{4}$-far from being $C_{2k+1}$-free.
	Then $G$ contains $\frac{\varepsilon^2}{4(2k+1)}n^2 \geq \frac{\varepsilon^2}{20k}n^2$ edge-disjoint copies of $C_{2k+1}$.
	Let $i \geq 1$ be maximal satisfying $\ell_i \leq \frac{4}{\varepsilon}$; this is well-defined because $\varepsilon \leq \frac{1}{4\ell_1}$ by assumption.
	Let $\ell := \ell_{i+1} = g(\ell_i)$. 
	By the maximality of $i$, we have $\ell_{i+1} > \frac{4}{\varepsilon} \geq 2k+1$.  
	We now apply Lemma \ref{lem:odd cycles} for the cycle lengths $C_{2k+1}$ and $C_{\ell}$ and with $\frac{\varepsilon^2}{20k}$ in place of $\varepsilon$. By this lemma, a sample of
	$q$ vertices of $G$ contains a copy of $C_{\ell}$ with probability at least $\frac{2}{3}$, where
	$$
	q = O\left( \frac{k^4\ell\log(\ell)}{\varepsilon^4} \right) = O(\varepsilon^{-8}) \cdot \ell \log(\ell) \leq O(\varepsilon^{-8}) \cdot g(4/\varepsilon) \cdot \log(g(4/\varepsilon)).
	$$
	Here we used that $\ell = g(\ell_i)$ and $\ell_i \leq 4/\varepsilon$. This completes the proof.
\end{proof}

We remark that Koml\'os \cite{Komlos} improved  \cref{lem:shortest odd cycle} by showing that every graph that is $\varepsilon$-far from bipartite contains an odd cycle of length $O(\varepsilon^{-1/2})$, which is best possible.
Using this result in place of  \cref{lem:shortest odd cycle} in the above proof, one can improve the bound in  \cref{thm:hierarchy} to
$w_L(\varepsilon) \leq O(\varepsilon^{-4}) \cdot g(\ell) \cdot \log(g(\ell))$, where $\ell = O(\varepsilon^{-1/2})$.
As was shown in \cite{GS}, this is essentially tight.

\section{Proof of Theorem \ref{thm:main-non-abundant}}\label{sec:non-abundant-proof}
In this section, we prove \cref{thm:main-non-abundant}, which states that there exist triangle-free tripartite non-$K_3$-abundant graphs. To prove this, we need to construct a triangle-free and tripartite graph $H$, and a sequence of graphs $G$, such that $G$ is $\varepsilon$-far from being triangle-free and contains $\varepsilon^{\omega(1)} n^{\ab{V(H)}}$ copies of $H$, where $\omega(1)$ tends to infinity as $\varepsilon \to 0$.

Our proof naturally splits into three parts. First, we construct $G$: it is a variant of the Ruzsa--Szemer\'edi construction which gives super-polynomial bounds for the triangle removal lemma. In \cref{subsec:RS}, we recall the relevant notions from additive combinatorics and give the construction of $G$. In \cref{subsec:strongly-genus-1}, we define what it means for a graph $H$ to be \emph{strongly genus-one}: this is a combinatorial condition which guarantees that $G$ contains few copies of $H$, thanks to the number-theoretic structure of $G$. Finally, in \cref{subsec:pseudorandom}, we prove that a graph satisfying certain pseudorandomness conditions is strongly genus-one. With this, all that remains is showing that there exist triangle-free graphs satisfying these pseudorandomness conditions, for which we pick an appropriate random graph and remove one edge from each triangle. Verifying that this graph satisfies the pseudorandomness conditions consists of standard arguments in random graph theory, which we do in \cref{sec:random-graphs}.

\subsection{Additive combinatorics and the Ruzsa--Szemer\'edi construction} \label{subsec:RS}
In the introduction, we defined what it means for an equation to have genus one; the following definition is a natural extension of this to families of equations.
\begin{definition}
  Let $S$ be a set of $s$ equations with integer coefficients and $k$ variables, namely
  \[
    S = \left\{ \sum_{i=1}^k a_{i,j} x_i = 0: 1 \leq j \leq s \right\}.
  \]
  We assume that each equation in $S$ is translation-invariant, meaning that $\sum_{i=1}^k a_{i,j}=0$ for all $j \in [s]$.
  One says that $S$ \emph{has genus one} if for all $\varnothing \subsetneq T \subsetneq [k]$, there exists some $j \in [s]$ so that $\sum_{i \in T} a_{i,j} \neq 0$.
\end{definition}

Let $S$ be a set of equations with genus one on $k$ variables. As in the case of a single equation, we say that integers $y_1,\dots,y_k$ form a \emph{non-trivial solution to $S$} if $y_1,\dots,y_k$ are not all equal, and if they satisfy all equations in $S$. We denote by $r_S(m)$ the maximum size of a subset of $[m]$ containing no non-trivial solution to $S$.
A simple argument shows that \cref{conj:ruzsa} implies the corresponding result for sets of equations, as stated in the following lemma.
\begin{lemma}\label{prop:ruzsa-for-sets}
  Assume that \cref{conj:ruzsa} holds. If $S$ is a set of linear equations of genus one, then $r_S(m) \geq m^{1-o(1)}$.
\end{lemma}
\begin{proof}
  Let $S$ consist of $s$ equations $E_1,\dots,E_s$ in $k$ variables. Let $b_1,\dots,b_s$ be independent and uniformly random numbers from $[2^k]$, and let $E=b_1 E_1 + \dotsb + b_s E_s$. Let the coefficients of $E$ be $a_1,\dots,a_k$, where $a_i = \sum_{j=1}^s b_j a_{i,j}$.

  Fix a set $\varnothing \subsetneq T \subsetneq [k]$. We claim that the probability that $\sum_{i \in T} a_i=0$ is at most $2^{-k}$. Indeed, since $S$ has genus one, there exists some $j \in [s]$ so that $\sum_{i \in T} a_{i,j} \neq 0$. For any fixed values of $b_1,\dots,b_{j-1},b_{j+1},\dots,b_s$, there is at most one choice of $b_j \in [2^k]$ so that $\sum_{i \in T} a_i=0$. This shows that the probability that $\sum_{i \in T} a_i=0$ is at most $2^{-k}$. As there are fewer than $2^k$ such sets $T$, we conclude that with positive probability, $E$ has genus one. We now fix $b_1,\dots,b_s$ such that $E$ has genus one.

  We now observe that $r_S(m) \geq r_E(m)$. Indeed, a non-trivial solution to $S$ is also a non-trivial solution to $E$, so a set avoiding non-trivial solutions to $E$ necessarily also avoids non-trivial solutions to $S$. By \cref{conj:ruzsa}, we have that $r_E(m) \geq m^{1-o(1)}$, yielding the claim.
\end{proof}

Additionally, it is easy to see that \cref{conj:ruzsa} implies the following result for multiple sets of equations.
\begin{lemma}\label{lem:random-shift}
  Assume that \cref{conj:ruzsa} holds. Let $S_1,\dots,S_t$ be sets of equations, such that each $S_i$ has genus one. Then there exists $R \subseteq [m]$ with $\ab R \geq m^{1-o(1)}$ such that $R$ has no non-trivial solutions to $S_i$, for all $1 \leq i \leq t$.
\end{lemma}
\begin{remark}
  This is stronger than saying that the union of genus-one sets of equations also has genus one. Indeed, avoiding a solution to some set $S$ of equations means that for any non-trivial assignment of the variables, at least one equation in $S$ is not satisfied. Here, we are insisting that at least one equation \emph{in each $S_i$} is not satisfied.
\end{remark}
\begin{proof}[Proof of \cref{lem:random-shift}]
  By \cref{prop:ruzsa-for-sets}, for each $i$, there is an $R_i \subseteq[m]$ with $\ab{R_i} \geq m^{1-o(1)}$ containing no non-trivial solutions to $S_i$. Let $s$ be the sum of the absolute values of the coefficients in all the equations in $S_1,\dots,S_t$, and let $M=2sm$. Then if we view $R_i$ as a subset of the cyclic group $\mathbb{Z}/M\mathbb{Z}$, then $R_i$ still contains no non-trivial solutions to $S_i$, as $M$ was chosen sufficiently large to not introduce any ``wraparound'' solutions.

  Let $a_2,\dots,a_t$ be independent and uniformly random elements of $\mathbb{Z}/M\mathbb{Z}$, and let $$R=R_1 \cap (R_2+a_2) \cap \dotsb \cap (R_t +a_t).$$ As all equations are translation-invariant, $R$ has no non-trivial solutions to $S_i$ for all $1 \leq i \leq t$. Additionally, as $R \subseteq R_1$, we may view $R$ as a subset of $[m]$. Finally,
  \[
    \E[\ab R] = \sum_{x \in R_1} \prod_{i=2}^t \pr(x \in R_i+a_2) = \ab{R_1} \cdot \prod_{i=2}^t\frac{\ab{R_i}}{M} \geq m^{1-o(1)},
  \]
  using the facts that $\ab{R_i} \geq m^{1-o(1)}$ for each $i$ and that $M=2sm = O_S(m)$. Thus, there is some $R\subseteq [m]$ with the desired properties.
\end{proof}

Fix an integer $m$ and a set $R \subseteq [m]$. We define the \emph{Ruzsa--Szemer\'edi graph} $\RS(m,R)$ to be the following graph. It has three parts $\A,\B,\C$, each of which we identify with $[3m]$. The edges are given by
\[
  (a,b) \in \A \times \B: b-a \in  R \qquad (b,c) \in \B \times \C: c-b \in  R \qquad (c,a) \in \C \times \A : c-a \in 2 R.
\]
The following property of the Ruzsa--Szemer\'edi graph is well-known. We include the proof for completeness.
\begin{lemma}\label{lem:eps-far}
	Let $m$ be an integer and let $R \subseteq [m]$. Then $\RS(m,R)$ contains a collection of $m\ab R$ edge-disjoint triangles.
\end{lemma}
\begin{proof}
  For every $a \in [m]$ and $r \in R$, we have a triangle $(a,a+r, a+2r) \in \A \times \B \times \C$. Given any two vertices in this triangle, we can recover the values of $a$ and $r$, so these triangles are edge-disjoint. As they are indexed by pairs $(a,r) \in [m] \times R$, there are $m \ab R$ such triangles.
\end{proof}
\subsection{Strongly genus-one graphs}\label{subsec:strongly-genus-1}
We now define what it means for a graph $H$ to be strongly genus-one. This is a Ramsey-theoretic condition that implies that a certain set of equations $S$ associated to the cycles in $H$ has genus one. Using this, we can show that an appropriate Ruzsa--Szemer\'edi graph $\RS(m,R)$ contains few copies of $H$, by choosing $R \subseteq [m]$ to be a set of integers containing no non-trivial solutions to $S$.

\begin{definition}[tagged cycle]
  Let $H$ be a tripartite graph with tripartition $A \cup B \cup C$, and suppose that the edges of $H$ are colored white or black. We say that a cycle $x_1,x_2,\dots,x_\ell,x_1$ in $H$ is \emph{tagged} if one of the following two conditions hold:
  \begin{itemize}
    \item Some edge $x_i x_{i+1}$ is black and the remaining edges are white (or this holds after swapping the colors).
    \item Two consecutive edges $x_{i-1} x_i$ and $x_i x_{i+1}$ are both black, all remaining edges are white, and $x_{i-1},x_i,x_{i+1}$ all lie in different parts $A,B,C$ (or this holds after swapping the colors).
  \end{itemize}
\end{definition}
Recall that a graph is \emph{uniquely 3-colorable} if it has a unique partition into three independent sets.
\begin{definition}[strongly genus-one graph]
  We say that a graph $H$ is \emph{strongly genus-one} if it is uniquely 3-colorable and  if, no matter how we color the edges of $H$ white or black such that there is at least one edge of each color, there is a tagged cycle.
\end{definition}
Note that the definition of a tagged cycle depends on the tripartition $A \cup B \cup C$. However, the definition of strongly genus-one requires $H$ to be uniquely 3-colorable, removing this ambiguity.

The Ruzsa--Szemer\'edi construction gives us a way of associating a set of equations $S$ to a tripartite graph $H$, in such a way that copies of $H$ in $\RS(m,R)$ are related to solutions of $S$ in $R$. We now define this set of equations.

Let $H$ be a tripartite graph with tripartition $A \cup B \cup C$. We assign every edge from $A$ to $B$ weight $1$, every edge from $B$ to $C$ weight $1$, and every edge from $C$ to $A$ weight $-2$. Note that we view the edges of $H$ as oriented, so an edge going from $B$ to $A$ receives weight $-1$, for example. For every edge $e \in E(H)$, we introduce a variable $x_e$. Every cycle in $H$ now naturally gives us an equation whose variables are $x_e$ for edges $e$ in the cycle, and whose coefficients are $\pm 1, \pm 2$: we simply write down the (signed) weights we see as we follow the cycle around the graph.
Formally, given a cycle $v_1,\dots,v_{\ell}$ with (oriented) edges $e_i = v_iv_{i+1}$, $i = 1,\dots,\ell$, we get the equation $\sum_{i=1}^{\ell}w(e_i) \cdot x_{e_i} = 0$, where $w(e_i)$ is the signed weight of $e_i$.
We now form a set of equations $S_{A,B,C}(H)$ by including all such equations, one equation for every cycle in the graph. Note that, as is reflected in the notation, this set of equations depends on $H$ and on the tripartition $A \cup B \cup C$.

The definition of tagged cycles implies that if $H$ is strongly genus-one, then $S_{A,B,C}(H)$ is a set of equations of genus one, as stated in the following lemma.

\begin{lemma}\label{lem:genus-one}
  Let $H$ be a strongly genus-one graph. Then for every tripartition\footnote{Of course, up to isomorphism, there is only one tripartition of $H$. But the set of equations $S_{A,B,C}(H)$ depends on the specific labeling we use of the sets, as this determines which edges receive which weights.} $A \cup B \cup C$ of $H$, the set of equations $S_{A,B,C}(H)$ has genus one.
\end{lemma}
\begin{proof}
  Suppose for contradiction that there is some set $\varnothing\subsetneq T \subsetneq E(H)$ so that in every equation in $S_{A,B,C}(H)$, the sum of the coefficients on the edges in $T$ is zero. Color the edges in $T$ black, and all remaining edges white. As we assume $\varnothing\subsetneq T \subsetneq E(H)$, we have at least one edge of each color. So by the definition of a strongly genus-one graph, we must have at least one tagged cycle. By swapping the colors if necessary\footnote{Note that if $T$ is such a ``bad'' set, then so is its complement, so we may swap the colors.}, we find a cycle with either exactly one black edge, or two consecutive black edges going between two different pairs of parts. Consider the equation corresponding to this cycle.

  In case there is exactly one black edge, then the sum of the coefficients on the edges in $T$ is $\pm 1$ or $\pm 2$. In particular, this sum is non-zero, contradicting our assumption on $T$. Similarly, if there are two consecutive black edges going between two different pairs of parts of $H$, then the sum of the coefficients on edges in $T$ is $\pm (1+1)=\pm 2$ or $\pm (-2+1)=\pm 1$, both of which are non-zero, another contradiction.
\end{proof}
\begin{remark}
  The definition of a tagged cycle is stronger than what is strictly needed for \cref{lem:genus-one}. Indeed, all we need is a cycle in $H$ such that the sum of the weights on the black edges in the cycle is non-zero. However, our construction of a strongly genus-one graph $H$ naturally yields a tagged cycle, rather than a more general ``non-canceling cycle''. This is the reason we use the term \emph{strongly} genus-one for such graphs: the combinatorial condition we require is stronger than what is needed simply to guarantee that $S_{A,B,C}(H)$ is a genus-one set of equations.
\end{remark}

We now have all the pieces in place to prove that strongly genus-one graphs are not $K_3$-abundant, assuming that \cref{conj:ruzsa} holds. 
\begin{lemma}\label{thm:superpoly}
  Assume that \cref{conj:ruzsa} holds. \nopagebreak
  
  Let $H$ be a strongly genus-one graph, and let $\varepsilon>0$ be sufficiently small. For every sufficiently large $N$, there exists an $N$-vertex graph $G$ which is $\varepsilon$-far from being triangle-free and which contains at most $\delta N^{v(H)}$ copies of $H$, where $\delta \leq \varepsilon^{\omega(1)}$ and the $\omega(1)$ tends to infinity as $\varepsilon \to 0$.

  In other words, $H$ is not $K_3$-abundant.
\end{lemma}
\begin{proof}
	As $H$ is uniquely 3-colorable, it has six tripartitions $A \cup B \cup C$ (namely the six permutations of $(A,B,C)$). By \cref{lem:genus-one}, each of the sets of equations $S_{A,B,C}(H)$ has genus one. So by \cref{lem:random-shift}, for every integer $m$, there is a set $R \subseteq [m]$ with $\ab R \geq m^{1-o(1)}$ which has no non-trivial solutions to any of these sets of equations.

  Let $m$ be the maximum integer so that there exists such a set $R$ of size $\ab R \geq 1000\varepsilon m$, and note that $m = \varepsilon^{-\omega(1)}$. Let $\Gamma=\RS(m,R)$. By \cref{lem:eps-far}, $\Gamma$ has $9m$ vertices and contains a collection of $m\ab R \geq 1000 \varepsilon m^2$ edge-disjoint triangles, so $\Gamma$ is $\varepsilon$-far from being triangle-free.

  The key claim is that there are at most $18m^2$ homomorphisms $H \to \Gamma$. Indeed, note that a homomorphism $H \to \Gamma$ in particular yields a proper 3-coloring of $H$, and since $H$ is uniquely 3-colorable, there are exactly six such colorings. Fix one of them, and say that it maps the parts $A,B,C$ of $H$ into the parts $\A,\B,\C$ of $\Gamma$, respectively. Every edge $e$ of $H$ gets mapped into some edge of $\Gamma$, and edges of $\Gamma$ are naturally labeled by elements of $R$. So this homomorphism gives us an assignment of the variables $\{x_e\}_{e \in E(H)}$ to values in $R$. Moreover, as every cycle in $H$ must map to a closed walk in $\Gamma$, we see that for every cycle in $H$, the corresponding equation in $S_{A,B,C}(H)$ is satisfied by this assignment of the variables.

  But we assumed that $R$ has no non-trivial solutions to $S_{A,B,C}(H)$, so we must have that every edge of $H$ is actually assigned to the same value of $R$. But having fixed this value, we see that it corresponds to at most $3m$ homomorphisms: once we determine where in $\A$ to send a single vertex of $A$, the locations of all remaining vertices are uniquely determined by the connectivity of $H$ ($H$ is connected since it is uniquely 3-colorable). So in total, there are at most $6\cdot 3m\ab R \leq 18m^2$ homomorphisms $H \to \Gamma$. Note that we used the fact that $R$ has no non-trivial solutions to \emph{any} of the six sets of equations arising from permuting $A,B,C$, so the argument above works regardless of how $\{A,B,C\}$ are mapped into $\{\A,\B,\C\}$ (and the union over all these contributions gives us the extra factor of six in the count of homomorphisms).

  To conclude the proof, let $N$ be sufficiently large, and suppose for simplicity that $N$ is a multiple of $9m$, say $N=9mt$. Let $G$ the balanced blowup $\Gamma[t]$, which has $N$ vertices. $G$ is still $\varepsilon$-far from being triangle-free. Indeed, suppose we delete fewer than $\varepsilon N^2$ edges from $G$. Sample a random copy of $\Gamma$ in $G$ by picking one uniformly random vertex from each blowup part of $G$. In expectation, we delete fewer than $\varepsilon \ab{V(\Gamma)}^2$ from this copy of $\Gamma$, so there exists a copy of $\Gamma$ in $G$ from which we deleted fewer than $\varepsilon \ab{V(\Gamma)}^2$ edges. As $\Gamma$ is $\varepsilon$-far from being triangle-free, this yields a triangle in $G$ which survives the edge-deletion, proving that $G$ is also $\varepsilon$-far from being triangle-free.

  Moreover, every homomorphism $H \to \Gamma$ yields at most $t^{\ab{V(H)}}$ copies of $H$ in $G$. So the total number of copies of $H$ in $G$ is at most
  \[
    18m^2 \cdot t^{\ab{V(H)}} = \frac{18m^2}{(9m)^{\ab{V(H)}}} N^{\ab{V(H)}} \eqqcolon \delta N^{\ab{V(H)}},
  \]
  where $\delta = O(m^{2-\ab{V(H)}}) \leq O(1/m) = \varepsilon^{\omega(1)}$, as claimed.
\end{proof}

\subsection{Pseudorandom graphs are strongly genus-one} \label{subsec:pseudorandom}

Given \cref{thm:superpoly}, all that remains to prove \cref{thm:main-non-abundant} is to show that there exist triangle-free strongly genus-one graphs.
In fact, we show that any tripartite graph satisfying appropriate pseudorandomness conditions is strongly genus-one; the main difficulty is finding a set of pseudorandomness conditions which are satisfied by an appropriate random graph, and which are strong enough to imply that the graph is strongly genus-one. The following proposition summarizes this set of pseudorandomness conditions. In a tripartite graph with parts $A,B,C$, given $X \subseteq A, Y \subseteq B$, let $N(X,Y)$ be the set of all vertices in $C$ with at least one neighbor in $X$ and at least one neighbor in $Y$.
For $Z \subseteq C$, let $N_A(Z)$ denote the set of vertices in $A$ adjacent to at least one vertex of $Z$.

\begin{proposition}\label{thm:pseudorandom}
  For all sufficiently large $n$, there exists a tripartite graph $H$ with parts $A,B,C$ with $\ab A = \ab B = \ab C=n$, having the following properties.
  \begin{propenum}
    \item If $F$ is a subgraph of $H[A\cup B]$ with at least half the edges of $H[A \cup B]$, then there exist $A' \subseteq A, B' \subseteq B$ with $F[A' \cup B']$ connected and $\ab{A'},\ab{B'} \geq n/10$. \label{it:connected}
    \item For all subsets $X \subseteq A, Y \subseteq B, Z \subseteq C$ of order at least $n/10$, we have that $\ab{N(X,Y)} \geq 99n/100$, and similarly for $\ab{N(X,Z)}$ and $\ab{N(Y,Z)}$. \label{it:common-nbhd}
    \item For all subsets $X \subseteq A, Y \subseteq B, Z \subseteq C$ of order at most $n/12$, we have that $\ab{N_A(Z)}>2\ab Z$, and similarly for all other pairs.\label{it:expansion}
    \item $H$ is triangle-free.\label{it:triangle-free}
    \item Up to permuting the colors, $H$ has a unique proper $3$-coloring.\label{it:uniquely-colorable}
  \end{propenum}
\end{proposition}
We prove \cref{thm:pseudorandom} by considering a random tripartite graph with edge density $p=n^{-3/4}$, and then deleting one edge from each triangle; we defer the proof to \cref{sec:random-graphs}. We remark that many other pseudorandom graphs would work; for example, there is nothing special about $p=n^{-3/4}$, and edge density $p=n^{-1+\theta}$ for any $0<\theta<1/3$ would work as well. As another example, the tripartite triple cover of Alon's explicit pseudorandom triangle-free graph \cite{Alon94} gives an explicit family of graphs satisfying the properties of \cref{thm:pseudorandom}.

The key claim, which completes the proof of \cref{thm:main-non-abundant}, is that this set of pseudorandomness conditions is enough to guarantee that $H$ is strongly genus-one. 

\begin{lemma}\label{thm:pseudorandom-implies-genus-one}
  If $H$ satisfies the conditions of \cref{thm:pseudorandom}, then $H$ is strongly genus-one.
\end{lemma}
From \cref{thm:pseudorandom,thm:pseudorandom-implies-genus-one}, we find that there exists a triangle-free strongly genus-one graph. Combined with \cref{thm:superpoly}, this implies that there exists a triangle-free tripartite graph which is not $K_3$-abundant (assuming \cref{conj:ruzsa}), which proves \cref{thm:main-non-abundant}.

The rest of this section is dedicated to proving \cref{thm:pseudorandom-implies-genus-one}.
So suppose we are given a tripartite graph $H$ which satisfies the conditions of \cref{thm:pseudorandom}, and we wish to prove that $H$ is strongly genus-one. By \cref{it:uniquely-colorable}, we know that $H$ is uniquely 3-colorable, so we only need to prove that in any two-coloring of $E(H)$ with at least one edge of each color, there is a tagged cycle.

The following are two useful lemmas allowing us to find tagged cycles.
\begin{lemma}\label{lem:same-part}
  Let $H$ be a tripartite graph with parts $A,B,C$, and suppose that the edges of $H$ are colored white or black.
  Suppose that $X \subseteq A, Y \subseteq B$ are such that $X\cup Y$ lies in a connected component of the white subgraph of $H$. If there is no tagged cycle, then every edge in $X \times Y$ is white.
\end{lemma}
\begin{proof}
  Suppose for contradiction that there is a black edge $xy$. Pick a white path connecting $x$ and $y$, which exists since we assumed that $X\cup Y$ lies in a connected component of the white graph. This yields a cycle with exactly one black edge, which is tagged.
\end{proof}
\begin{lemma}\label{lem:other-part}
  Let $H$ be a tripartite graph with parts $A,B,C$, and suppose that the edges of $H$ are colored white or black.
  Suppose that $X \subseteq A, Y \subseteq B$ are such that the $X\cup Y$ lies in a connected component of the white subgraph of $H$. Let $Z=N(X,Y)\subseteq C$. If there is no tagged cycle, then every edge in $(X \times Z) \cup (Y \times Z)$ is white.
\end{lemma}
\begin{proof}
  Suppose for contradiction that there is a black edge $yz \in Y \times Z$. By the definition of $Z = N(X,Y)$, we know that $z$ has some neighbor $x \in X$. Pick a white path $x = v_0, v_1, \dots, v_k=y$ connecting $x$ and $y$. Suppose first that $z$ is not on this path, i.e.\ $z \notin \{v_1,\dots,v_k\}$. Then we have a cycle $x, v_1, \dots, v_{k-1}, y, z, x$. In this cycle, the edge $yz$ is black, the edge $zx$ may be black or white, and all remaining edges are white. Regardless of the color of $zx$, we obtain a tagged cycle.

  On the other hand, suppose that $v_i = z$ for some $i$. Note that $i < k-1$ as the edge $yz$ is black and we assumed that the edge $v_{k-1} y$ is white. This implies that we get a cycle $z, v_{i+1},\dots,v_{k-1}, y, z$. In this cycle, the edge $yz$ is black and the remaining edges are white, so we again find a tagged cycle.

  The same argument, upon interchanging the roles of $X$ and $Y$, shows that no edge $xz \in X \times Z$ can be black.
\end{proof}

We are now ready to complete the proof. The basic idea is to repeatedly use \cref{lem:other-part,lem:same-part}, as well as the pseudorandomness conditions guaranteed by \cref{thm:pseudorandom} to gradually improve our understanding of the coloring, until we eventually show that if there is no tagged cycle, then all the edges must be of the same color.
\begin{proof}[Proof of \cref{thm:pseudorandom-implies-genus-one}]
Let $H$ be a graph satisfying the conditions of \cref{thm:pseudorandom}, and suppose for contradiction that there is a coloring of $H$ with no tagged cycle. We will show that $H$ is colored monochromatically.

Suppose without loss of generality that at least half the edges in $H[A \cup B]$ are white. By \cref{it:connected}, there exist $A_1 \subseteq A, B_1 \subseteq B$ such that the white graph on $H[A_1 \cup B_1]$ is connected, with $\ab{A_1},\ab{B_1} \geq n/10$. Let $C_1 = N(A_1,B_1)$. By \cref{it:common-nbhd}, we know that $\ab{C_1} \geq 99n/100$. Moreover, by \cref{lem:other-part,lem:same-part}, we have that $H[A_1 \cup B_1 \cup C_1]$ is white, since the connectivity of the white graph on $H[A_1 \cup B_1]$ implies that $A_1 \cup B_1$ lies in a connected component of the white subgraph.

Now, let $A_2 = N(B_1,C_1)$. Note that $B_1 \cup C_1$ lies in a connected component of the white subgraph, since every vertex in $C_1$ has at least one neighbor in $B_1$, and since $B_1$ lies in a connected component of the white subgraph. Therefore, $H[A_2 \cup B_1 \cup C_1]$ is white by \cref{lem:same-part,lem:other-part}. As $\ab{B_1},\ab{C_1} \geq n/10$, we also see that $\ab{A_2} \geq 99n/100$ (again by \cref{it:common-nbhd}). Similarly, letting $B_2 = N(A_2,C_1)$, we find that $\ab{B_2}\geq 99n/100$ and $H[A_2 \cup B_2 \cup C_1]$ is white.

Now, let $(A',B',C')$ be a maximal triple of sets so that $H[A' \cup B' \cup C']$ is white. By the above, we know that $\ab {A'}, \ab {B'}, \ab {C'} \geq 99n/100$. We claim that in fact, $\ab {A'}= \ab {B'}= \ab {C'}=n$, meaning that $H$ is monochromatically white. So suppose this is not the case. Without loss of generality, assume that $\ab{C'} \leq\ab {A'}, \ab {B'} $, and let $W = C \setminus C'$.

Suppose there is a vertex $w \in W$ with at least one neighbor in $A'$ and at least one neighbor in $B'$. Then $w \in N(A',B')$, so by \cref{lem:other-part}, all edges in $\{w\} \times (A' \cup B')$ are white. This means that we may add $w$ to $C'$, contradicting maximality of $(A',B',C')$. So every vertex in $W$ is non-adjacent to at least one of $A',B'$. Without loss of generality, there is $W' \subseteq W$ with $\ab{W'} \geq \frac 12 \ab W$ so that every vertex in $W'$ is non-adjacent to $A'$. This means that $N_A(W') \subseteq A \setminus A'$. But by \cref{it:expansion} (which we may apply since $\ab{W'} \leq \ab W \leq n/100 \leq n/12$),
\[
  \ab{C \setminus C'} =\ab W \leq 2 \ab{W'} < \ab{N_A(W')} \leq \ab{A \setminus A'},
\]
which contradicts our assumption that $\ab{A'} \geq \ab{C'}$. This contradiction completes the proof.
\end{proof}

\section{Proof of Proposition \ref{thm:pseudorandom}}\label{sec:random-graphs}
To complete the proof of \cref{thm:main-non-abundant}, it remains to prove \cref{thm:pseudorandom}. The various parts of \cref{thm:pseudorandom} are essentially independent of one another, so the proof more or less breaks down into a number of lemmas about random graphs, which we now state and prove. Recall that an event $E$ is said to happen \emph{with high probability (w.h.p.)}\ if $\pr(E) \to 1$ as $n \to \infty$, where the implicit parameter $n$ will always be clear from context.

\begin{lemma}\label{lem:connected}
  Let $H$ be a random bipartite graph with vertex set $A \cup B$, where $\ab A = \ab B=n$, and with edge density $p=n^{-3/4}$. The following holds w.h.p.\ as $n \to \infty$. If $F\subseteq H$ satisfies $e(F) \geq \frac 13 e(H)$, then there exist $A' \subseteq A, B' \subseteq B$ with $\ab{A'},\ab{B'} \geq n/10$ such that $F[A' \cup B']$ is connected.
\end{lemma}
\begin{proof}
  By the Chernoff bound, $H$ has at least $\frac 34 pn^2$ edges w.h.p. We now condition on this event happening.

  Suppose for contradiction that there exists an $F$ without this property, and let its connected components be $F_1,\dots,F_\ell$. For $1 \leq i \leq \ell$, let the vertex set of $F_i$ be $A_i \cup B_i$, where $A_i \subseteq A, B_i \subseteq B$. By assumption, we know that $\min\{\ab{A_i},\ab{B_i}\}< n/10$ for each $i$, and that
  \[
    \frac 14 pn^2 \leq \frac 13 e(H) \leq e(F) = \sum_{i=1}^\ell e_F(A_i,B_i) \leq \sum_{i=1}^\ell e_H(A_i,B_i).
  \]
  So to upper-bound the probability that such an $F$ exists, it suffices to upper-bound the probability that there exist disjoint $A_1,\dots,A_\ell \subseteq A, B_1,\dots,B_\ell \subseteq B$ with the properties that $\min\{\ab{A_i},\ab{B_i}\}< n/10$ for each $i$ and that $\sum_i e_H(A_i,B_i) \geq \frac 14 pn^2$.

  Fix disjoint sets $A_1,\dots,A_\ell \subseteq A, B_1,\dots,B_\ell \subseteq B$ with $\min\{\ab{A_i},\ab{B_i}\}< n/10$ for each $i$. Let $I$ be the set of indices $i$ with $\ab{A_i} \leq n/10$, and let $J=[\ell] \setminus I$. We have that
  \[
    \sum_{i=1}^\ell \ab{A_i} \ab{B_i} = \sum_{i \in I} \ab{A_i} \ab{B_i} + \sum_{i \in J} \ab{A_i} \ab{B_i} \leq \frac n{10} \sum_{i \in I} \ab{B_i} + \frac n {10} \sum_{i \in J} \ab{A_i} \leq \frac{n}{10}(\ab{B}+\ab A) = \frac{n^2}{5}.
  \]
  Let $X$ denote the random variable $\sum_{i=1}^\ell e_H(A_i,B_i)$. By the computation above, we see that $X$ is upper-bounded in distribution by $\Bin(n^2/5, p)$. By the Chernoff bound,
  \[
    \pr\left(X \geq \frac 14 pn^2\right) \leq \pr\left(\Bin\left(\frac{n^2}{5},p\right)\geq \frac{pn^2}{4}\right) \leq \exp\left(-\frac{pn^2}{250}\right).
  \]
  There are at most $n^n$ partitions of an $n$-element set\footnote{Much more precise bounds are known on the number of partitions of $[n]$, also known as the \emph{Bell numbers}. The simple upper bound of $n^n$ follows by noting that every function $[n] \to [n]$ yields a partition of  the domain.}, so this error probability of $e^{-\Omega(pn^2)}=e^{-\Omega(n^{5/4})}$ is enough to union-bound over at most $n^{2n} = e^{O(n\log n)}$ choices for $A_1,\dots,A_\ell,B_1,\dots,B_\ell$.
\end{proof}

Recall that for $X \subseteq A, Y \subseteq B$, we denote by $N(X,Y)$ the set of all vertices in $C$ with at least one neighbor in $X$ and at least one neighbor in $Y$.

\begin{lemma}\label{lem:NXY}
  Let $H$ be a tripartite random graph on parts $A,B,C$ with $\ab A = \ab B = \ab C = n$, where each edge appears with probability $p=n^{-3/4}$. The following holds w.h.p.\ as $n \to \infty$. For all $X \subseteq A,Y\subseteq B$  with $\ab{X},\ab{Y} \geq n/10$, we have that $\ab{N(X,Y)} \geq n-o(n)$.
\end{lemma}
\begin{proof}
  Fix $X,Y$. For $c \in C$, let $E_c$ be the event that $c$ has no neighbor in $X$ or no neighbor in $Y$. Then
  \[
    \pr(E_c) \leq (1-p)^{\ab{X}} + (1-p)^{\ab{Y}} \leq 2 (1-p)^{n/10} \leq 2 e^{-n^{1/4}/10} \leq e^{-n^{1/5}}
  \]
  for all sufficiently large $n$.

  Note that the events $E_c$ are independent for different choices of $c$. Therefore, for an integer $k$, the probability that $E_c$ happens for at least $k$ different choices of $c$ is at most
  \[
    \binom nk \pr(E_c)^k \leq n^k e^{-kn^{1/5}} \leq e^{-k n^{1/6}}
  \]
  since $n^{1/5}-\log n \geq n^{1/6}$ for sufficiently large $n$. Now let $k=n^{7/8}$, so that the probability that this happens is at most $e^{-n^{25/24}}$. This is enough to beat the union bound over all $2^{2n}$ choices of $X \subseteq A, Y \subseteq B$. In particular, we find that for every such pair, $\ab{N(X,Y)} \geq n - k=n-o(n)$.
\end{proof}

\begin{lemma}\label{lem:expansion}
  Let $H$ be a bipartite random graph on parts $A,B$ with $\ab A = \ab B = n$, where each edge appears with probability $p=n^{-3/4}$. The following holds w.h.p.\ as $n \to \infty$. For all non-empty $W\subseteq B$  with $\ab{W} \leq n/12$, we have that $\ab{N_A(W)}> 6\ab W$.
\end{lemma}
\begin{proof}
  Summing over all choices for $w=\ab W$, the probability that all edges from $W$ go to another set of size at most $6w$ is at most
  \[
    \sum_{w=1}^{n/12} \binom nw \binom n{6w} (1-p)^{w(n-6w)} \leq \sum_{w=1}^{n/12} n^{7w} e^{-pwn/2} = \sum_{w=1}^{n/12} \exp\left(w\left(7\log n- \frac 12 n^{1/4}\right)\right) = o(1).\qedhere
  \]
\end{proof}

\begin{lemma}\label{lem:triangles}
  Let $H$ be a tripartite random graph on parts $A,B,C$ with $\ab A = \ab B = \ab C = n$, where each edge appears with probability $p=n^{-3/4}$. Then w.h.p.\ there are at most $n^{4/5}$ triangles, and every vertex lies in at most four triangles.
\end{lemma}
\begin{proof}
  The expected number of triangles in $H$ is $p^3 n^3 = n^{3/4}$, so the first result follows by Markov's inequality.

  For the second, first note that by the Chernoff bound, w.h.p.\ every vertex has degree at most $2pn = 2n^{1/4}$. We now condition on this event. There is a bijection between triangles containing a vertex $v$ and edges inside $N(v)$, so it suffices to prove that w.h.p.\ $N(v)$ spans at most four edges. For fixed $v$, the probability that $e(N(v))\geq 5$ is at most
  \[
    \binom{\binom{\ab{N(v)}}2}5 p^5 \leq \binom{2n^{1/2}}5 p^5 \leq 32 (n^{1/2}p)^5 = 32 n^{-5/4}.
  \]
  By the union bound, the probability that this happens for any vertex $v$ is at most $O(n^{-1/4})=o(1)$.
\end{proof}
\begin{lemma}\label{lem:big-sets}
  Let $H$ be a bipartite random graph on parts $A,B$ with $\ab A = \ab B = n$, where each edge appears with probability $p=n^{-3/4}$.
  The following holds w.h.p.\ as $n \to \infty$. For all $X \subseteq A, Y \subseteq B$ with $\ab X, \ab Y \geq n/100$, we have that $e(X,Y) \geq n$.
\end{lemma}
\begin{proof}
  The expected number of edges between $X$ and $Y$ is $p\ab X \ab Y \geq pn^2/10^4 \geq 2n$ for sufficiently large $n$. Therefore, by the Chernoff bound, the probability that $e(X,Y)<n$ is at most $e^{-pn^2/10^5} = e^{-\Omega(n^{5/4})}$. This is small enough to union-bound over all $\leq 2^{2n}$ choices for $X,Y$.
\end{proof}

We are finally ready to prove \cref{thm:pseudorandom}.
\begin{proof}[Proof of \cref{thm:pseudorandom}]
  We sample a random tripartite graph $H_0$ on vertex set $A \cup B \cup C$, where each edge appears with probability $p=n^{-3/4}$. With high probability, the outcomes of \cref{lem:connected,lem:expansion,lem:NXY,lem:triangles,lem:big-sets} hold. We now delete a single edge from every triangle in $H_0$. In so doing, we delete at most $n^{4/5}$ edges, and for any vertex, we delete at most four of its incident edges, both by \cref{lem:triangles}. Let $H$ be the resulting graph. We claim that $H$ satisfies all the desired properties.

  Certainly, $H$ is triangle-free, proving \cref{it:triangle-free}. For \cref{it:expansion}, note that every vertex in $Z$ lost at most four incident edges when passing from $H_0$ to $H$, so $\ab{N_A(Z)}$ decreased by at most $4 \ab Z$ when passing from $H_0$ to $H$. As \cref{lem:expansion} gives that $\ab{N_A(Z)}>6\ab Z$ in $H_0$, we get \cref{it:expansion}.  For \cref{it:common-nbhd}, recall that we delete at most $n^{4/5}$ edges when passing from $H_0$ to $H$, so for any fixed sets $X,Y$, the value of $\ab{N(X,Y)}$ decreases by at most $n^{4/5}$ when passing from $H_0$ to $H$. As $\ab{N(X,Y)} \geq n-o(n)$ in $H_0$ by \cref{lem:NXY}, we get \cref{it:common-nbhd}. Finally, for \cref{it:connected}, note that if $F$ has at least half the edges of $H[A \cup B]$, then it has at least one-third the edges of $H_0[A \cup B]$, so the result follows immediately from \cref{lem:connected}.

  We now note that for all subsets $X\subseteq A, Y \subseteq B$ of order at least $n/100$, there is at least one edge between $X$ and $Y$. Indeed,
  \cref{lem:big-sets} guarantees at least $n$ edges in $H_0$ between $X$ and $Y$, and at most $n^{4/5}$ of these are deleted when passing to $H_0$, so at least one remains. Similarly, for any $Z \subseteq C$ of order at least $n/100$, there is at least one edge between $X$ and $Z$ and at least one edge between $Y$ and $Z$.

  It remains to prove \cref{it:uniquely-colorable}. So suppose that there is a proper 3-coloring of $H$, namely a partition into independent sets $D,E,F$. We wish to prove that up to permuting the colors, $A=D, B=E, C=F$. We first claim that at most one of the numbers $\ab{D \cap A}, \ab{D \cap B}, \ab{D \cap C}$ is at least $n/100$. Indeed, suppose without loss of generality that $\ab{D \cap A}, \ab{D \cap B} \geq n/100$. Then by the observation in the previous paragraph, there is an edge between $D \cap A$ and $D \cap B$, contradicting that $D$ is an independent set.

  Using this claim, we see that, potentially after permuting the colors, we have that $\ab{D \cap A}, \ab{E \cap B}, \ab{F \cap C} \geq \frac{98}{100}n$. We now claim that in fact $D=A, E=B, F=C$. Indeed, suppose this is not the case, and assume without loss of generality that $\ab{C \setminus F} \geq \ab{A \setminus D}, \ab{B \setminus E}$. Without loss of generality, at least half the vertices in $C \setminus F$ are colored $D$. Let $Z = (C \setminus F) \cap D$. Note that $\ab Z \leq \ab{C \setminus F} \leq n/50$, so by \cref{it:expansion}, we have that $\ab{N_A(Z)}>2 \ab Z \geq \ab{C \setminus F} \geq \ab{A \setminus D}$, as we assumed that $\ab Z \geq \frac 12 \ab{C \setminus F}$. Since $\ab{N_A(Z)} > \ab{A \setminus D}$, there must be at least one edge between $Z$ and $A \cap D$. But every vertex in $Z$ and in $A \cap D$ is colored $D$, contradicting that $D$ is an independent set. This contradiction concludes the proof of \cref{it:uniquely-colorable}.
\end{proof}

\section{Concluding remarks}\label{sec:conclusion}
A fundamental question which remains open is \cref{prob:K_t-abundant}: does there exist a $t$-chromatic $K_t$-abundant graph for any $t \geq 4$? We are inclined to believe that this is not the case, and every $t$-chromatic graph is non-$K_t$-abundant when $t \geq 4$. As mentioned in the introduction, the first open case is the Brinkmann graph; it would be very interesting to prove or disprove that this graph is $K_4$-abundant. 

Another interesting question is whether one can prove \cref{thm:main-non-abundant} unconditionally, i.e.\ without relying on \cref{conj:ruzsa}. 
It is known \cite[Theorem 2.3]{Ruzsa} that \cref{conj:ruzsa} holds in case $E$ has exactly one negative coefficient, as the Behrend construction can be used to find large sets with no non-trivial solutions to such equations; such equations are called \emph{convex}. Even simple cases beyond this are open; for example, if $E$ is the equation $x+3y=2z+2w$, then the best known lower bound on $r_E(m)$ is $\Omega(\sqrt m)$.

The sets of equations $S_{A,B,C}(G)$ that we obtain in the proof of \cref{thm:main-non-abundant} are very complicated, and they involve many coefficients of both signs. However, in the proof of \cref{prop:ruzsa-for-sets}, we have a great deal of freedom in how we construct a single equation from our family of equations. It is conceivable that  one could combine $S_{A,B,C}(G)$ into a single convex equation. If this were possible, it would yield a proof of \cref{thm:main-non-abundant} without the assumption that \cref{conj:ruzsa} holds.

However, computer simulation suggests that this may not work. Indeed, we wrote a computer program to sample a random tripartite graph and delete one edge from every triangle. Testing whether the set of equations arising from the cycles linearly spans a convex equation can be done efficiently, as this can be described as a linear programming problem. Even when testing fairly large random graphs (with hundreds of vertices and thousands of edges), we were not able to find one where the equations arising from cycles span a convex equation.

\paragraph{Acknowledgments}
We are grateful to the anonymous referees for their helpful suggestions.


\begin{thebibliography}{10}
\providecommand{\url}[1]{\texttt{#1}}
\providecommand{\urlprefix}{URL }
\providecommand{\eprint}[2][]{\url{#2}}

\bibitem{Alon94}
N.~Alon, Explicit {R}amsey graphs and orthonormal labelings, \emph{Electron. J.
  Combin.} \textbf{1} (1994), Research Paper 12, 8pp.

\bibitem{Alon}
N.~Alon, Testing subgraphs in large graphs, \emph{Random Structures Algorithms}
  \textbf{21} (2002), 359--370.

\bibitem{AK}
N.~Alon and M.~Krivelevich, Testing {$k$}-colorability, \emph{SIAM J. Discrete
  Math.} \textbf{15} (2002), 211--227.

\bibitem{AS}
N.~Alon and A.~Shapira, A characterization of easily testable induced
  subgraphs, \emph{Combin. Probab. Comput.} \textbf{15} (2006), 791--805.

\bibitem{Behrend}
F.~A. Behrend, On sets of integers which contain no three terms in arithmetical
  progression, \emph{Proc. Nat. Acad. Sci. U.S.A.} \textbf{32} (1946),
  331--332.

\bibitem{BESS}
B.~Bollob\'{a}s, P.~Erd\H{o}s, M.~Simonovits, and E.~Szemer\'{e}di, Extremal
  graphs without large forbidden subgraphs, \emph{Ann. Discrete Math.}
  \textbf{3} (1978), 29--41.

\bibitem{CF}
D.~Conlon and J.~Fox, Graph removal lemmas, in \emph{Surveys in combinatorics
  2013}, \emph{London Math. Soc. Lecture Note Ser.}, vol. 409, Cambridge Univ.
  Press, Cambridge, 2013,  1--49.

\bibitem{CFSZ}
D.~Conlon, J.~Fox, B.~Sudakov, and Y.~Zhao, The regularity method for graphs
  with few 4-cycles, \emph{J. Lond. Math. Soc. (2)} \textbf{104} (2021),
  2376--2401.

\bibitem{Csaba}
B.~Csaba, Regular decomposition of the edge set of graphs with applications,
  2021. Preprint available at arXiv:2109.12394.

\bibitem{DLR}
R.~A. Duke, H.~Lefmann, and V.~R\"{o}dl, A fast approximation algorithm for
  computing the frequencies of subgraphs in a given graph, \emph{SIAM J.
  Comput.} \textbf{24} (1995), 598--620.

\bibitem{Erdos}
P.~Erd\H{o}s, On extremal problems of graphs and generalized graphs,
  \emph{Israel J. Math.} \textbf{2} (1964), 183--190.

\bibitem{Fox}
J.~Fox, A new proof of the graph removal lemma, \emph{Ann. of Math. (2)}
  \textbf{174} (2011), 561--579.

\bibitem{FZ}
J.~Fox and Y.~Zhao, Removal lemmas and approximate homomorphisms, \emph{Combin.
  Probab. Comput.} \textbf{31} (2022), 721--736.

\bibitem{FK}
A.~Frieze and R.~Kannan, Quick approximation to matrices and applications,
  \emph{Combinatorica} \textbf{19} (1999), 175--220.

\bibitem{GS}
L.~Gishboliner and A.~Shapira, A generalized {T}ur\'{a}n problem and its
  applications, \emph{Int. Math. Res. Not. IMRN}  (2020), 3417--3452.

\bibitem{GGR}
O.~Goldreich, S.~Goldwasser, and D.~Ron, Property testing and its connection to
  learning and approximation, \emph{J. ACM} \textbf{45} (1998), 653--750.

\bibitem{Komlos}
J.~Koml\'{o}s, Covering odd cycles, \emph{Combinatorica} \textbf{17} (1997),
  393--400.

\bibitem{KST}
T.~K\"{o}vari, V.~T. S\'{o}s, and P.~Tur\'{a}n, On a problem of {K}.
  {Z}arankiewicz, \emph{Colloq. Math.} \textbf{3} (1954), 50--57.

\bibitem{Ruzsa}
I.~Z. Ruzsa, Solving a linear equation in a set of integers. {I}, \emph{Acta
  Arith.} \textbf{65} (1993), 259--282.

\bibitem{RS}
I.~Z. Ruzsa and E.~Szemer\'{e}di, Triple systems with no six points carrying
  three triangles, in \emph{Combinatorics ({P}roc. {F}ifth {H}ungarian
  {C}olloq., {K}eszthely, 1976), {V}ol. {II}}, \emph{Colloq. Math. Soc.
  J\'{a}nos Bolyai}, vol.~18, North-Holland, Amsterdam-New York, 1978,
  939--945.

\end{thebibliography}
\end{document}